\documentclass[12pt]{article}
\usepackage{amsmath,amssymb,amsfonts}
\usepackage{graphicx,psfrag,epsf}
\usepackage{enumerate}
\usepackage{url} 

\newcommand{\blind}{1}

\def\ov{\overline}

\newtheorem{theorem}{Theorem}[section]
\newtheorem{proposition}[theorem]{Proposition}
\newtheorem{definition}[theorem]{Definition}
\newtheorem{lemma}[theorem]{Lemma}

\newtheorem{example}[theorem]{Example}
\newtheorem{counterexample}[theorem]{Counterexample}

\newtheorem{remark}[theorem]{Remark}
\newenvironment{proof}{\emph{Proof.}} {\quad \hfill $\blacksquare$ \newline}

\usepackage{color}
\usepackage{soul} 

\begin{document}

\def\spacingset#1{\renewcommand{\baselinestretch}%
{#1}\small\normalsize} \spacingset{1}


{
  \date{}
  \title{\bf New stochastic comparisons based on Tail Value at Risk measures}
  \author{F\'elix Belzunce\thanks{
    The research of  F\'elix Belzunce is funded by the Ministerio de Econom\'ia, Industria y Competitividad (Spain) under grant \textit{MTM2016-79943-P} (AEI/ FEDER, UE).}\hspace{.2cm}\\
   Dpto. Estad\'istica e Investigaci\'on Operativa, Universidad de Murcia
\\ Facultad de Matem\'aticas, Campus de Espinardo \\
30100 Espinardo (Murcia), SPAIN\\ \texttt{belzunce@um.es}\\
    and\\
    Alba M. Franco-Pereira\thanks{
    Alba M. Franco-Pereira acknowledges support received from the Ministerio de Econom\'ia y Competitividad (Spain) under grant \textit{MTM2017-89422-P} and has received financial support from the Xunta de Galicia (Centro Singular de Investigación de Galicia accreditation 2016-2019) and the European Union (European Regional Development Fund - ERDF). She also acknowledges funding from Banco Santander and Complutense University of Madrid (project \textit{PR26/16-5B-1})}\hspace{.2cm}\\
    Instituto de Matemática Interdisciplinar (IMI) and\\ Dpto. Estadística e Investigación Operativa,\\ Universidad Complutense de Madrid \\ Facultad de Ciencias, Plaza Ciencias 3 \\ 28040
Madrid, SPAIN\\ 
\texttt{albfranc@ucm.es}\\
    and \\
    Julio Mulero\thanks{
    Julio Mulero wants to acknowledge the support received from the Conselleria d'Educaci\'o, Investigaci\'o, Cultura i Esport (Generalitat de la Comunitat Valenciana) under grant \textit{GV/2017/015} and the Ministerio de Econom\'ia, Industria y Competitividad (Spain) under grant \textit{MTM2016-79943-P} (AEI/ FEDER, UE).}\hspace{.2cm}\\
    Dpto. Matem\'aticas, Universidad de
Alicante \\ Facultad de Ciencias, Apartado de correos 99 \\ 03080
Alicante, SPAIN \\ \texttt{julio.mulero@ua.es}}
  \maketitle
  \newpage
} 

\if1\blind
{
  \bigskip
  \bigskip
  \bigskip
  \begin{center}
    {\LARGE\bf New stochastic comparisons based on tail values at risk}
\end{center}
  \medskip
} \fi

\bigskip
\begin{abstract}
 In this paper we provide a new criterion for the comparison of claims, when we have conditional claims arising in  stop loss contracts or contracts with franchise deductible. These stochastic comparisons are made on the basis of the Tail Value at Risk (also known as conditional tail expectation), just for a fixed level and beyond. In particular, we explain the interest of comparing these quantities, study some preservation properties and, in addition, we provide sufficient conditions for its study. Finally we illustrate its usefulness with some examples.
\end{abstract}

\noindent%
\vfill

\newpage
\spacingset{1.45} 
\section{Introduction and motivation} Let $X$ be a random risk  that represents, for example, the future claims of an insurance company with distribution function $F$. The Value at Risk of $X$ for some prescribed confidence level $p\in(0, 1)$ (or, simply, quantile at $p$), is denoted as $\text{VaR}[X;p]$ (or $F^{-1}(p)$ for the sake of simplicity), and given by
\begin{equation}\label{quantile}
\text{VaR}[X;p]=\inf\{x: F(x)\geq p\},\text{ for all }p\in(0,1).
\end{equation}
In this context, this value is the maximum claim which can occur with $100p\%$ confidence over a certain period of time; see, for instance, H\"urlimann (2002, 2003).  The VaR is nowadays the most commonly used measure of risk, but it does not give any information about the claim amount beyond the VaR level. However, the Value at Risk is not a coherent risk measure as it does not satisfy the subadditivity property. An alternative to the VaR, given in H\"urlimann (2003) is the Tail Value at Risk. It is denoted as $\text{TVaR}$ and defined by
\[
\text{TVaR}\left[ X;p\right] =\dfrac{1}{1-p}\int_p^1 \text{VaR}[X;u]du, \text{ for all } p\in[0,1).
\]
It can be interpreted as the average claim an insurance company will suffer in case of (extreme) situations where claims exceed the predefined confidence level $p\in(0,1)$. It is well known that the Tail Value at Risk is a coherent risk measure. For continuous random variables, it holds that
\[
\mathrm{TVaR}\left[ X;p\right] =\mathrm{E}[X\mid X>\text{VaR}[X;p]],
\]
which is also known as the conditional tail expectation or expected shortfall of $X$. It is important to note that the Basel Committee on Banking Supervision in January 2016 wishes to carry out a key reform consisting of:  
\begin{quote}
``a shift from Value-at-Risk (VaR) to an Expected Shortfall (ES) measure of risk under stress. The use of ES will help to ensure a more prudent capture of ``tail risk'' and capital adequacy during periods of significant financial market stress''.
\end{quote}
See Embrechts \textit{et al}. (2014) for further information about this discussion.

The Value at Risk and the Tail Value at Risk are  used in the theory of stochastic orders to provide stochastic comparisons between random risks. For example, if $Y$ is another random risk with quantile $\text{VaR}[Y;p]$, and it holds that
\[
\text{VaR}[X;p]\leq \text{VaR}[Y;p],\text{ for all }p\in(0,1),
\]
then $X$ is said to be smaller than $Y$ in the usual \textit{stochastic order}, denoted as $X\leq_{st}Y$. 

Another well-known stochastic comparison based on the Tail Value at Risk is the increasing convex order which holds if, and only if,
\begin{equation}\label{icxdefinition}
\int_{p}^{1}\text{VaR}[X;u]du \leq \int_{p}^{1}\text{VaR}[Y;u]du,\text{ for all }p\in(0,1).
\end{equation}

In this case, $X$ is said to be smaller than $Y$ in the \textit{increasing convex order} and it is denoted as $X\leq_{icx}Y$ (see, for example, Berrendero and C\'arcamo, 2012). Two general references for further information on stochastic orderings are Shaked and Shanthikumar (2007) and Belzunce, Martínez-Riquelme and Mulero (2016).  

However, in some situations we are not faced with the comparison of the whole random claims $X$, but with residual random claims. Next, we describe two situations where these residual claims appear. 

Let us consider a company confronted with a risky business over some time period, and let the random variable $X$ represent the claims that the company incurs at the end of the period. Suppose now that the company insures itself against heavy claims, that is, against claims above some deductible $t$. Then, the claim that the reinsurer experiences (if it does) is known as residual claims at time $t$ and is given by
\[
X_t=[X-t\mid X>t],\text{ for all }t<u_X,
\]
 where $u_X$ is the right endpoint of the support of $X$. This is called a \textit{stop-loss contract} with deductible $t$.  
 
The residual claims  also arise when a franchise deductible is incorporated to the contract. In this case, the insurer pays only the part of the claim which exceeds the amount $t$ and if the size of the claim falls below this amount, then the claim is not covered by the contract and the insured
receives no indemnification. Finally, in life insurance these conditional random variables appears as the residual life for a person who has survived up to time $t$, and are usually represented in terms of life tables. 

 If we denote the survival function of $X$ as $\ov F(t)=P[X>t]$, for all $t$ in the support of $X$, the survival function of $X_t$ is given by
\[
\ov F_{t}(x)=\frac{\ov F(x+t)}{\ov F(t)},\text{ for all }x>0,
\]
and is well defined for all $t<u_X$, even if $t$ is not in the support of $X$, and is of interest not only in insurance issues but also in other many areas of applied probability and statistics such as actuarial studies, biometry, survival analysis, economics and reliability.  A useful tool to provide properties of $X$ is the expectation of $X_t$, which is called the mean residual life function,  and is given by
\[
m(t)=E[X-t\mid X>t],
\]
for all $t<u_X$ (and 0, elsewhere) provided the expectation exists.  In insurance this value represents the expected claims from $t$, \textit{i.e.}, the expected cost of the claims from $t$ assumed for the reinsurance company.

Some stochastic orders have been defined in terms of comparisons between residual claims. On one hand, if 
\[
X_t\leq_{st}Y_t,\text{ for all }t\text{ such that }\ov F(t), \ov G(t)>0, 
\]
then $X$ is said to be smaller than $Y$ in the \textit{hazard rate order}, denoted as $X\leq_{hr}Y$.  On the other hand, if 
\begin{equation}\label{mrldef}
X_t\leq_{icx}Y_t,\text{ for all }t, 
\end{equation}
then $X$ is said to be smaller than $Y$ in the \textit{mean residual life order}, denoted as $X\leq_{mrl}Y$. In the continuous case, \eqref{mrldef} is related to the comparison of the mean residual life functions, that is, they are equivalent to  
\[
m(t)\leq l(t),\text{ for all }t, 
\]
where $l$ be the mean residual life function of a random variable $Y$ (for further information, see Shaked and Shanthikumar, 2007).

Besides the survival function and mean of residual claims, some other quantities of interest have been considered in the literature.  Given $p\in(0,1)$, an alternative to the mean residual life is the percentile residual life function which is defined by $\text{VaR}[X_t;p]$, for all $t<u_X$. For instance,  the median residual life function is obtained when $p=0.5$ (see Lillo, 2005, for a complete study of the median residual life function and related aging notions).

Franco-Pereira \textit{et al.} (2011) considered the comparison between the percentile residual life functions of two random variables and proposed the family of the \textit{percentile residual life orders}, denoted as $p\textup{-rl}$ order. In particular, for $p\in(0,1)$, given two random variables $X$ and $Y$ with percentile residual life functions $\text{VaR}[X_t;p]$ and $\text{VaR}[Y_t;p]$, respectively, then  
\begin{equation}\label{prl}
X\leq_{p\textup{-rl}}Y\Leftrightarrow \text{VaR}[X_t;p]\leq \text{VaR}[Y_t;p],\text{ for all }t.
\end{equation}

An interesting property in Franco-Pereira \textit{et al.} (2011) is the comparison between the TVaR's in their Theorem 5.1. More specifically, if $X$ and $Y$ are two random variables such that $X\leq_{p-rl}Y$,  for all $q\in[p,1)$, then 
\[
\int_q^1 \text{VaR}[X_t;u]du\leq \int_q^1 \text{VaR}[Y_t;u]du,
\]
or, equivalently, $\text{TVaR}[X_t;p]\leq \text{TVaR}[Y_t;p]$. 

As we have noticed previously, it is more interesting from an applied and a theoretical point of view, to provide comparisons in terms of the tail values at risk instead of the values at risk. Also, in contrast with the criteria provided by Franco-Pereira \textit{et al.} (2011), it is more reasonable to provide a comparison not for a fixed level $p_0$ but for $p_0$ and beyond. 

From previous considerations, in this paper we propose a new
family of stochastic orders indexed by $p_0\in(0,1)$, called in general the $p_0$-tvar-rl order, that is based on the comparison of the tail values at risk of the residual claims from $p_0$ and beyond and for all deductibles $t$. The paper is organized as follows. In Section \ref{sec:2}, we define the $p_0\text{-tvar-rl}$ order and study their main properties, relationships with other known stochastic orders through different examples and counterexamples and provide some closure results. Sufficient conditions for this order to hold are stated in  Section \ref{sec:4} and applications to some theoretical examples are given. In Section \ref{sec:5}, we illustrate the previous examples with real datasets. Finally, in Section \ref{sec:6}, we give a brief summary of the work.

From now on, we will  follow the notation in Denuit \textit{et al.} (2005). Besides, given two random variables $X$ and $Y$, we denote the cumulative  distribution functions by $F$ and $G$,  the  survival function by $\ov F$ and $\ov G$, and  the corresponding values at risk by $F^{-1}$ and $G^{-1}$. Besides, we will denote the corresponding residual lives by  $X_t=[X-t\mid X>t]$ and $Y_t=[Y-t\mid X>t]$ and their values at risk by $F^{-1}_t$ and $G^{-1}_t$, respectively.

\section{Definition and properties of the $p_0\text{-tvar-rl}$ order}\label{sec:2}

From the motivations of the previous section, we provide the definition of the following criterion to compare risks. For the rest of the paper, we assume that the considered random variables have finite means, which ensures the existence of the measures considered in the comparisons.

\begin{definition}\label{definiciontvarl}
Let $X$ and $Y$ be two random variables with right endpoint of their supports $u_X$ and $u_Y$, respectively, and $p_0\in(0,1)$. We say that $X$ is less than $Y$ in the \emph{Tail value at Risk from $p_0$ of the residual life order}, denoted by $X\leq_{p_0\textup{-tvar-rl}} Y$, if 
\[
\mathrm{TVaR}[X_t;p]\leq \mathrm{TVaR}[Y_t;p], \text{ for all } p\in[p_0,1)\text{ and }t<u_X,u_Y,
\]
or, equivalently,
\[
\int_p^1 F^{-1}_t(u)du\leq \int_p^1 G^{-1}_t(u)du, \text{ for all } p\in[p_0,1)\text{ and }t<u_X,u_Y.
\]
\end{definition}

Given a random variable $X$ with distribution function $F$ and right endpoint of its support $u_X$, it is easy to see that
\begin{equation}\label{tvaryvar}
F^{-1}_t(p)=F^{-1}(p+(1-p)F(t))-t, \text{ for all } t<u_X \text{ and }p\in(0,1).
\end{equation}

Therefore,  $X \leq_{p_0\text{-tvar-rl}}Y$ if, and only if,
\begin{equation}\label{ec1}
\int_{p}^{1}F^{-1}(u+(1-u)F(t))du \leq \int_{p}^{1}G^{-1}(u+(1-u)G(t))du,
\end{equation}
for all $t<u_X,u_Y$ and $p\in [p_0,1)$, with $p_0 \in (0,1)$.

The $p_0$-tvar-rl order satisfies some desirable closure properties that are given in the following propositions. Recall that,  given a sequence of random variables $\{X_n:n=1,2,\dots\}$ with distribution functions $F_n$ and $F$, respectively, then $X_n$ is said to converge in distribution to $X$, denoted by $X_n\overset{\text{d}}{\longrightarrow} X$, if $\lim_{n\rightarrow +\infty}F_n(x)=F(x)$ for all $x$ at which $F$ is continuous.

\begin{theorem}[Closure under convergence in distribution]
Let $\{X_n:n=1,2,\dots\}$ and $\{Y_n:n=1,2,\dots\}$ be two sequences of positive continuous random variables such that $X_n\overset{\text{d}}{\longrightarrow} X$ and $Y_n\overset{\text{d}}{\longrightarrow} Y$. Assume that $X_n$ and $X$ have a common interval support for all $n\in\mathbb N$ and $\lim_{n \rightarrow \infty}E[X_n]=E[X]$ (and, analogously, for $Y_n$ and $Y$ and their expectations). If $X_n\leq_{p_0\textup{-tvar-rl}}Y_n$, then $X\leq_{p_0\textup{-tvar-rl}}Y$.
\end{theorem}

\begin{proof}
Given that $X_n\overset{\emph{d}}{\longrightarrow} X$ and $\lim_{n \rightarrow \infty}E[X_n]= E[X]$, we have that
\begin{equation}\label{convergence1}
\lim_{n\rightarrow+\infty}\int_{x}^{+\infty}\overline F_n(y)dy=\int_{x}^{+\infty}\overline F(y)dy,\text{ for all }x,
\end{equation}
where $F_n$ is the distribution function of $X_n$, for all $n\in\mathbb N$ (see M\"uller, 1996). The same reasoning holds also for $Y_n$ and $Y$.

Moreover, from $X_n\overset{\emph{d}}{\longrightarrow} X$, we have that
\begin{equation}
\lim_{n\rightarrow+\infty}F^{-1}_n(x)=  F^{-1}(x),\text{ for all }x\in(0,1),\label{convergence2}
\end{equation}
and
\begin{equation}
\lim_{n\rightarrow+\infty}F^{-1}_{n,t}(x)=  F^{-1}_t(x),\text{ for all }x\in(0,1),\label{convergence3}
\end{equation}
where $F_{n,t}$ ($G_{n,t}$) is the distribution function of $X_{n,t}=[X_n-t\mid X_n>t]$ ($Y_{n,t}=[Y_n-t\mid Y_n>t]$), for all $n\in\mathbb N$ (see van der Vaart, 1998, and Lemma 4.2 in Franco-Pereira et al., 2011, respectively). 

From \eqref{convergence1}, \eqref{convergence2}, \eqref{convergence3} and
\[
\int_{p}^{1}F^{-1}_{n,t}(u)du=\int_{F^{-1}_{n,t}(p)}^{+\infty}\frac{\overline F_{n}(t+y)}{\overline F_{n}(t)}dy+(1-p)F^{-1}_{n,t}(p),
\]
it follows that
\[
\lim_{n\rightarrow+\infty}\int_{p}^{1} F^{-1}_{n,t}(u)du=\int_{p}^{1} F^{-1}_{t}(u)du.
\]

Therefore, if $X_n\leq_{p_0\textup{-tvar-rl}}Y_n$, 
\begin{equation*}
\int_{p}^{1}F^{-1}_{t}(u)du = \lim_{n \rightarrow \infty}\int_{p}^{1}F^{-1}_{n,t}(u)du 
\leq  \lim_{n \rightarrow \infty}\int_{p}^{1}G^{-1}_{n,t}(u)du=\int_{p}^{1}G^{-1}_{t}(u)du,
\end{equation*}
for all $p\geq p_0$.
\end{proof}

Next, we give the closure under transformations. Recall the following lemma from Barlow and Proschan (1975).
\begin{lemma}[Barlow and Proschan, 1975, p. 120]\label{lemmaBP}
Let $W$ be a measure on the interval $(a,b)$, not necessarily nonnegative. Let $h$ be a nonnegative function defined on $(a,b)$.  If $\int_t^b dW(x)\geq 0$ for all $t\in (a,b)$ and if $h$ is increasing, then $\int_a^b h(x)dW(x)\geq 0$.
\end{lemma}

\begin{theorem}[Closure under transformations]\label{closuretrans}
Let $X$ and $Y$ be two random variables and $\phi$ be an increasing and convex function. Then, $X\leq_{p_0\textup{-tvar-rl}}Y$ if, and only if, $\phi(X)\leq_{p_0\textup{-tvar-rl}}\phi(Y)$.
\end{theorem}

\begin{proof}
For a random variable $X$ with distribution function $F$ and a strictly nondecreasing and continuous function $\phi$, let $\phi(X)_t$ be the residual life of $\phi(X)$, for $t<u_{\phi(X)}$. Then, the Value at Risk of $\phi(X)_t$, denoted as $F_{\phi,t}^{-1}$, is given by
\[
F_{\phi,t}^{-1}(u)=\phi(F^{-1}(u+(1-u)F(\phi^{-1}(t))))-t.
\]
Let $Y$ be another random variable with distribution function $G$ such that $X\leq_{p_0\textup{-tvar-rl}}Y$ for a certain $p_0\in(0,1)$. From Definition \ref{definiciontvarl} and \eqref{ec1}, we have
\begin{equation*}
\int_{p}^{1}F^{-1}(u+(1-u)F(t))du \leq \int_{p}^{1}G^{-1}(u+(1-u)G(t))du,
\end{equation*}
for all $t<u_X,u_Y$ ($\overline F(t),\overline G(t)>0$) and $p\in [p_0,1)$. Let $t':=\phi^{-1}(t)$ such that $\overline F(t'),\overline G(t')>0$, then we want to prove that 
\begin{equation*}
\int_{p}^{1}\phi(F^{-1}(u+(1-u)F(t')))du \leq \int_{p}^{1}\phi(G^{-1}(u+(1-u)G(t')))du,
\end{equation*}
or, equivalently,
\begin{equation}\label{inequalityclosure}
\int_{p}^{1}\left[\phi(G^{-1}(u+(1-u)G(t')))-\phi(F^{-1}(u+(1-u)F(t')))\right]du \geq 0,
\end{equation}
for all $p\in [p_0,1)$.

It is known that, given an increasing and convex function $\phi$, then $\phi$ is continuous and there exists a positive and increasing function $l$ such that
\begin{equation}\label{property1closuretrans}
\phi(b)-\phi(a)=\int_a^b l(v)dv.
\end{equation}

From \eqref{property1closuretrans}, the expression in \eqref{inequalityclosure} can be written as
\begin{multline}\label{barlowproschanclosure}
\int_{p}^{1}\left[\phi(G^{-1}(u+(1-u)G(t')))-\phi(F^{-1}(u+(1-u)F(t')))\right]du\\
=\int_p^1\int_{F^{-1}(u+(1-u)F(t'))}^{G^{-1}(u+(1-u)G(t'))} l(v)dv du.
\end{multline}

First, let us note a lower bound for \eqref{barlowproschanclosure}. On one hand, let us suppose that $F^{-1}(u+(1-u)F(t'))\leq G^{-1}(u+(1-u)G(t'))$, then it holds that
\begin{multline*}
\int_{F^{-1}(u+(1-u)F(t'))}^{G^{-1}(u+(1-u)G(t'))} l(v)dv\geq \\
l(F^{-1}(u+(1-u)F(t')))\left[G^{-1}(u+(1-u)G(t'))-F^{-1}(u+(1-u)F(t'))\right].
\end{multline*}
On the other hand, if $F^{-1}(u+(1-u)F(\phi^{-1}(t)))> G^{-1}(u+(1-u)G(\phi^{-1}(t)))$, then
\begin{multline*}
\int_{F^{-1}(u+(1-u)F(t'))}^{G^{-1}(u+(1-u)G(t'))} l(v)dv=-\int_{G^{-1}(u+(1-u)G(t'))}^{F^{-1}(u+(1-u)F(t'))} l(v)dv\geq \\
l(F^{-1}(u+(1-u)F(t')))\left[G^{-1}(u+(1-u)G(t'))-F^{-1}(u+(1-u)F(t'))\right].
\end{multline*}
Therefore, we have
\begin{multline*}
\int_p^1\int_{F^{-1}(u+(1-u)F(t'))}^{G^{-1}(u+(1-u)G(t'))} l(v)dv du\geq \\ \int_p^1 l(F^{-1}(u+(1-u)F(t')))\left[G^{-1}(u+(1-u)G(t'))-F^{-1}(u+(1-u)F(t'))\right]du.
\end{multline*}

Now, let us consider
\[
dW(u)=\left[G^{-1}(u+(1-u)G(t'))-F^{-1}(u+(1-u)F(t'))\right]du,
\]
and
\[
h(u)=l(F^{-1}(u+(1-u)F(t')))I(p<u).
\]
By hypothesis, we have that
\[
\int_p^1 dW(u)\geq 0,
\]
and, given that $l$ and $I(p<u)$ are nondecreasing functions, $h$ is also nondecreasing. Finally, from Lemma \ref{lemmaBP},
\[
\int_p^1 l(u)dW(u)\geq 0.
\]
\end{proof}

Next, we give relationships between the new order and some of the previously defined stochastic orders. Some of the proofs follow easily and are ommitted. 

\begin{proposition}\label{prlyp0tvarrl2}
Let $X$ and $Y$ be two random variables such that $X \leq_{p_0\textup{-tvar-rl}}Y$ for $p_0\in(0,1)$, then  $X \leq_{q\textup{-tvar-rl}}Y$, for all $q\in[p_0,1)$.
\end{proposition}

\begin{proposition}
Let $X$ and $Y$ be two random variables such that $X \leq_{p\textup{-rl}}Y$ for all $p\in[p_0,1)$, then  $X \leq_{p\textup{-tvar-rl}}Y$, for all $p\in[p_0,1)$.
\end{proposition}

\begin{proposition}\label{icxp0tvarrl}
Let $X$ and $Y$ be two random variables. Then, $X \leq_{mrl}Y$ if, and only if, $X \leq_{p_0\textup{-tvar-rl}}Y$ for all $p_0\in(0,1)$.
\end{proposition}

\begin{proof}
Let us recall that $X \leq_{mrl}Y$ is equivalent to $X_t\leq_{icx}Y_t$, $t<u_X,u_Y$ (see, for example, Belzunce, Mart\'inez-Riquelme and Mulero, 2016). From \eqref{icxdefinition}, if $X_t \leq_{icx}Y_t$, we have 
\[
\int_{p}^{1}F^{-1}(u+(1-u)F(t))du \leq \int_{p}^{1}G^{-1}(u+(1-u)G(t))du,\textup{ for all } p\in(0,1).
\]
The proof follows from (\ref{ec1}). On the other hand, $X \leq_{0\textup{-tvar-rl}}Y$ implies  $X_t \leq_{icx}Y_t$, for all $t<u_X,u_Y$, that is, $X\leq_{mrl}Y$. 
\end{proof}

Next, we show that the hazard rate order implies the $p_0\text{-tvar-rl}$ , for all $p_0\in(0,1)$. The proof follows easily by taking into account that $\leq_{hr}$ implies $\leq_{mrl}$.

\begin{proposition}\label{hr}
Let $X$ and $Y$ be two random variables such that $X\leq_{hr}Y$, then $X \leq_{p_0\text{-tvar-rl}}Y$,
for all $p_0 \in (0,1)$. 
\end{proposition}

Now it is natural to wonder if $X \leq_{p_0\text{-tvar-rl}}Y$, for $p_0\in(0,1)$, implies $X \leq_{hr}Y$. Unfortunately, the answer is negative as we can see in the following counterexample which is based on Remark 3.7 and Counterexample A.1 in Franco-Pereira \textit{et al.} (2011). 

\begin{counterexample}\label{contraejemplo1}
 Let us fix $p_0 \in (0,1)$ and consider $X(p_0)$ with distribution function given by
\[
F_{p_0}(t)=
\begin{cases}
0, &  t<p_0,\\
t, & p_0 \leq t<1,\\
1, & t\geq 1,\\
\end{cases}
\]
that is, $X(p_0)$ is a mixture of a uniform distribution on $(p_0,1)$ with probability $1-p_0$, and a degenerate random variable at $p_0$ with probability $p_0$. Let us denote as $F^{-1}_{p_0,t}$ the corresponding Value at Risk of $X(p_0)_t$, which is the residual life of $X(p_0)$. Now, let $Y$ be a uniformly distributed on $(0,1)$ random variable and let $G^{-1}_t$ be the Value at Risk of $Y_t$. For any $p \in [p_0,1)$, it can be seen that
\[
F^{-1}_{p_0,t}(p)=
\begin{cases}
p-t, &  t<p_0,\\
p(1-t), & p_0 \leq t<1,\\
0, & t\geq 1,\\
\end{cases}
\]
and
\[
G^{-1}_t(p)=
\begin{cases}
p-t, &  t<0,\\
p(1-t), & 0 \leq t<1,\\
0, & t\geq 1.\\
\end{cases}
\]

It can be seen that $X(p_0) \leq_{p\textup{-rl}}Y$, for all $p \in [p_0,1)$, which implies $X(p_0) \leq_{p_0\text{-tvar-rl}}Y$. However, $X(p_0) \nleq_{hr}Y$. Note also that $E[X(p_0)]=\frac{p_0^2+1}{2}>E[Y]=\frac{1}{2}$, then the $p_0\textup{-tvar-rl}$ order does not preserve expectations. 
\end{counterexample}

In light of the previous counterexample, any stochastic order that preserves expectations cannot be implied  by the $p_0\textup{-tvar-rl}$ order. In particular, the $p_0\textup{-tvar-rl}$ order does not imply the stochastic, the increasing convex and the mean residual life orders. 

To finish with the study of the possible relations between this new criteria and other orders, we show that the usual stochastic order does not imply the $p_0\textup{-tvar-rl}$ order, as we can see in the following counterexample. 

\begin{counterexample}\label{contraejemplo2}
Under the assumptions of Counterexample \ref{contraejemplo1}, let $F_{p_0}$ and $G$ be the distribution functions of $X(p_0)$ and $Y$, respectively. It is easy to see that $F_{p_0}(t)\leq G(t)$ for all $t$, \textit{i.e.}, $Y \leq_{st}X(p_0)$. However, $Y \nleq_{p_0\textup{-tvar-rl}}X(p_0)$. In fact, if $t \in (0,p_0)$, we have that
\begin{equation*}\label{contraej2}
\int_{p}^{1}F^{-1}_{p_0,t}(u)du < \int_{p}^{1}G^{-1}_{t}(u)du.
\end{equation*}
In fact, the previous condition holds if, and only if,
\begin{eqnarray*}
\int_{p}^{1}(u-t)du < \int_{p}^{1}u(1-t)du &\Leftrightarrow& \int_{p}^{1} (ut-t)du=-2t(p-1)^2<0,
\end{eqnarray*}
which is always true.

In Counterexample A.2 in Franco-Pereira \textit{et al.} (2011) it is shown that, for any $p \in (0,1)$, the mean residual life order does not imply the $p\textup{-rl}$ order. Since $X\nleq_{p\textup{-rl}} Y$, we have that $X\nleq_{p_0\textup{-tvar-rl}}Y$ for any $p_0 \leq p$. 
\end{counterexample}

Finally, we give the following relationship among the new order and the icx order of certain random variables.

\begin{proposition}\label{icxsufficienteconditions}
Let $X$ and $Y$ be two random variables with finite means and right endpoint of their supports $u_X$ and $u_Y$, respectively. If $\max\left\{X_t,F_t^{-1}(p_0)\right\}\leq_{icx} \max\left\{Y_t,G_t^{-1}(p_0)\right\}$, for all $t<u_X,u_Y$, then $X\leq_{p_0\textup{-tvar-rl}}Y$.
\end{proposition}

\begin{proof}
Given a random variable $X$ with distribution function $F$, let us consider $X_p=\max\left\{X,F^{-1}(p)\right\}$, for all $p\in(0,1)$ and let $F^{-1}_p$ be its corresponding Value at Risk at $p$. Then, it is easy to see that
\begin{equation}\label{quantileXp}
F^{-1}_p(q)=\left\{\begin{array}{ll}
F^{-1}(p), &\text{ if }0<q<p,\\
F^{-1}(q), &\text{ if }p\leq q<1.
\end{array}\right.
\end{equation}

From \eqref{quantileXp}, the expectation of $X_p$ is given by
\begin{eqnarray}\label{meanXp}
E[X_p]=\int_0^1 F^{-1}_p(u)du&=&\int_0^p F^{-1}(p)du+\int_p^1 F^{-1}(u)du\nonumber\\
&=&p F^{-1}(p)+\int_p^1 F^{-1}(u)du.
\end{eqnarray}

Furthermore, the Tail Value at Risk of $X_p$ is given by
\[
\text{TVaR}[X_p;q]=\left\{\begin{array}{ll}
\frac{1}{1-q}\left(\int_p^1F^{-1}(u)du+F^{-1}(p)(p-q)\right), &\text{ if }0<q<p,\\
\frac{1}{1-q}\int_q^1 F^{-1}(u)du=\text{TVaR}[X;q], &\text{ if }p\leq q<1.
\end{array}\right.
\]

Now, let $X$ and $Y$ be two random variables with distribution function $F$ and $G$, respectively. From \eqref{icxdefinition}, we have that $X_p\leq_{icx}Y_p$, for $p\in (0,1)$, if, and only if, the two following two conditions hold
\begin{itemize}
\item[(i)] $\int_p^1F^{-1}(u)du+F^{-1}(p)(p-q)\leq \int_p^1G^{-1}(u)du+G^{-1}(p)(p-q)$, for all $0<q<p$, and
\item[(ii)] $\text{TVaR}[X;q]\leq \text{TVaR}[Y;q]$, for all $p\leq q<1$.
\end{itemize}

Therefore, we have $\text{TVaR}[X;q]\leq \text{TVaR}[Y;q]$, for all $p\leq q<1$. 
\end{proof}

\section{On sufficient conditions for the $p_0\textup{-tvar-rl}$ order}\label{sec:4}

Let $X$ and $Y$ be two random variables with distribution functions $F$ and $G$, and  right endpoint of their supports $u_X$ and $u_Y$, respectively. For $t<u_X,u_Y$, let
\begin{eqnarray*}
H_t(p)&=&G^{-1}_t(p)-F^{-1}_t(p)\\
&=& G^{-1}(p+(1-p)G(t))-F^{-1}(p+(1-p)F(t)),
\end{eqnarray*}
for all $p\in[0,1)$. Given $p_0\in (0,1)$, from Definition \ref{definiciontvarl} and \eqref{ec1}, $X\leq_{p_0\textup{-tvar-rl}} Y$ if, and only if,
\begin{equation}\label{intHt}
L_t(p)=\int_p^1 H_t(u)du\geq 0,
\end{equation}
for all $t<u_X, u_Y$ and $p\in[p_0,1)$.

It is apparent that in order to have the $p_0\textup{-tvar-rl}$ order, it is sufficient to assure that $L_t(p)\geq 0$, from a certain $p_0$ and for all $t<u_X,u_Y$.

For a fixed $t<u_X, u_Y$, we have that
\begin{eqnarray}\label{lt0}
L_t(0)&=&\mathrm{E}[Y_t]-\mathrm{E}[X_t],
\end{eqnarray}
which, obviously, can be positive  or negative. In particular, if it is positive [negative] for all $t$, then $X\leq_{mrl}[\geq_{mrl}]Y$.
Moreover, it holds that
\begin{eqnarray}\label{lt1}
\lim_{p\rightarrow 1^-} L_t(p)&=&0.
\end{eqnarray}

Next, we provide sufficient conditions for the $p_0\textup{-tvar-rl}$ order of two random variables. 
Following Karlin (1968), p. 20, let $f$ be a real function defined on $I \subseteq R$, the number of sign changes of
$f$ in $I$ is defined by 
\[
S^-(f(x)) = \sup\{S^-[f(x_1),\dots,f(x_m)]\}
\]
where $S^-[y_1,\dots,y_m]$ is the number of sign changes of the indicated sequence, zero terms
being discarded, and the supremum is extended over all sets $x_1 < x_2 < \dots< x_m\in I$, and for
all $m < +\infty$. 
Given two left continuous functions $f$ and $g$ defined on
an interval $I\subseteq R$, we will say
that the point $x_0 \in I$ is a crossing point of $f$ and $g$, if there exists $\varepsilon_1, \varepsilon_2 > 0$, such that $f (x) - g(x) > [<]0$, for all $x \in (x_0,x_0 + \varepsilon_1)$ and $f (x) - g(x) \leq [\geq]0$, for all $x \in (x_0 - \varepsilon_2,x_0)$, with strict
inequality at some point. Recall that quantile functions are left continuous.

\begin{proposition}\label{propcondsuf}
Let $X$ and $Y$ be two random variables with finite means and right endpoint of their supports $u_X$ and $u_Y$, respectively. For a fixed $t<u_X, u_Y$, let us assume that $S^-(G_t^{-1}-F_t^{-1})\geq 1$ 
where
the last sign change occurs from $-$ to + with crossing point $p_t \in (0, 1)$ (which is known as the up-crossing point). Then, $L_t(p)\geq 0$, (at least) for all $p\geq p_t$.
\end{proposition} 

\begin{proof}
The proof follows easily by taking into account that $L_t(p)=\int_p^1 H_t(u)du\geq 0$, for all $p\in[p_t,1)$, if $H_t(p)\leq 0$, for all $p\in[p_t,1)$.
\end{proof}

It is worth to mention that if Proposition \ref{propcondsuf} holds, then it could exist another $p'_t<p_t$ such that $L_t(p)\geq 0$, for all $p\geq p'_t$. 

From now on, let $u=\min\{u_X,u_Y\}$ and $T=(-\infty,u)$, i.e., the set of all possible $t$'s; $T_1$, the set of all $t$'s such that $F^{-1}_t(p)\leq G^{-1}_t(p)$, for all $p\in(0,1)$; and $T_2$ the set of all $t$'s for which Proposition \ref{propcondsuf} holds, \textit{i.e.}, it exists $p_t\in(0,1)$ such that $L_t(p)\geq 0$, for all $p\geq p_t$. 

\begin{theorem}\label{teocondsuf}
Let $X$ and $Y$ be two random variables with finite means and right endpoint of their supports $u_X$ and $u_Y$, respectively. If $T=T_1\ \cup\ T_2$, then $X\leq_{p_0\textup{-tvar-rl}} Y$, for some $p_0\in(0,1)$. More specifically, $X\leq_{p_0\textup{-tvar-rl}} Y$ for
\[
p_0=\max\{p_t:t\in T_2\}.
\]
\end{theorem}

\begin{proof}
Fixed $t<u_X, u_Y$, if $t\in T_1$, then $L_t(p)\geq 0$, for all $p\in(0,1)$. Otherwise, if $t\in T_2$ holds, then $L_t(p)\geq 0$, for all $p\in(p_t,1)$. Therefore, $L_t(p)\geq 0$ for all $p_0\in(0,1)$ where $p_0=\max\{p_t:0<t<u_X, u_Y\}$ and we have that $X\leq_{p_0\textup{-tvar-rl}} Y$.
\end{proof}

\begin{remark}\label{remark3}
 Given two random variables $X$ and $Y$ with distribution functions $F$ and $G$, respectively, Theorem \ref{teocondsuf} holds if, for each $t\in T=(-\infty,u_X)$, it holds one of the following sets of conditions:
\begin{itemize}
\item[(i)] $F^{-1}(p+(1-p)G(t))\leq G^{-1}(p+(1-p)F(t))$, for all $p\in(0,1)$ (\textit{i.e.}, $t\in T_1$), or
\item[(ii)] $S^-(G^{-1}(p+(1-p)G(t))-F^{-1}(p+(1-p)F(t)))\geq 1$ where the last sign change occurs from - to + with crossing point $p_t \in (0, 1)$ (\textit{i.e.}, $t\in T_2$).
\end{itemize} 
Therefore,  the bivariate plot of
\begin{equation}\label{bivariateplot}
\begin{array}{ccc}
H^*:(0,1)\times (-\infty,u_X)&\mapsto& \mathbb R\\
(p,t) &\rightsquigarrow & G^{-1}(p+(1-p)G(t))-F^{-1}(p+(1-p)F(t)),
\end{array}
\end{equation}
could be a feasible approach to analyze the changes of sign of $G_t^{-1}-F_t^{-1}$, for each $t$. This plot can be implemented using the \textsf{VaRES} package in \textsf{R} (\url{www.r-project.org}, R Core Team, Vienna, Austria) which provides all the quantiles and distribution functions (see Nadarajah, Chan and Afuecheta, 2017, for details).
\end{remark}

To finish, the following examples illustrate Theorem \ref{teocondsuf} for different models  of interest in risk theory such as  the Pareto, the loglogistic and the distorted distributions. 

\begin{example}[Pareto distributions]\label{pareto2}
The Pareto distribution is a well-known suitable model for many nonnegative socioeconomic variables. For example, this model is one of the most important
mathematical models for calculating excess of loss premiums in risk theory. There are several univariate versions of the original distribution proposed by Vilfredo Pareto in 1897. A random risk follows a generalized Pareto distribution, denoted by $\textup{GPD}(\xi,\mu,\sigma)$ for $\xi,\mu\in\mathbb R$ and $\sigma>0$, if its survival function is given by
\[
\ov F(x)=\left\{\begin{array}{ll}
\left(1+\xi\frac{x-\mu}{\sigma}\right)^{-1/\xi}, & \xi\neq 0,\\
\exp\left(-\frac{x-\mu}{\sigma}\right), & \xi=0, 
\end{array} \right.
\]
for all $x\geq \mu$, if $\xi\geq 0$, and $\mu\leq x\leq \mu-\sigma/\xi$, if $\xi<0$  (see, for example, Arnold, 2014 and 2015). In particular, given $\xi>0$ and $t\in\mathbb R$, it holds that
\begin{eqnarray*}
\overline F_t(x) &=& \frac{\left(1+\xi\frac{x+t-\mu}{\sigma}\right)^{-1/\xi}}{\left(1+\xi\frac{t-\mu}{\sigma}\right)^{-1/\xi}},\text{ for all }t<u_X,x\in\mathbb R,\\
F^{-1}_t(p) &=& \frac{\sigma}{\xi}\left[(1-p)^{-\xi}\left(1+\xi\frac{t-\mu}{\sigma}\right)-1\right]-t+\mu,\text{ for all }t<u_X,p\in\mathbb R.
\end{eqnarray*}

Let $X\sim\textup{GPD}(\xi,\mu_X,\sigma_X)$ and $Y\sim\textup{GPD}(\xi,\mu_Y,\sigma_Y)$ with $\xi\neq 0$. After a straightforward computation, it follows that
\begin{eqnarray*}
H_t(p) & = & G^{-1}_t(p)-F^{-1}_t(p)\\
&=& (1-p)^{-\xi}\left(\frac{\sigma_Y}{\xi}-\frac{\sigma_X}{\xi}+\mu_X-\mu_Y\right)-\left(\frac{\sigma_Y}{\xi}-\frac{\sigma_X}{\xi}+\mu_X-\mu_Y\right),
\end{eqnarray*}
for all $p\in(0,1)$ and $t<u_X,u_Y$. 

It is easy to see that
\begin{eqnarray*}
H_t(0) &=& 0,\\
H'_t(p) &=& \xi(1-p)^{-\xi-1}\left(\frac{\sigma_Y}{\xi}-\frac{\sigma_X}{\xi}+\mu_X-\mu_Y\right).
\end{eqnarray*}
Therefore, $F^{-1}_t(p)\leq G^{-1}_t(p)$, for all $p\in(0,1)$ and $t\in\mathbb R$, whenever $\xi>[<]0$ and
$\sigma_Y-\sigma_X>[<]\xi(\mu_Y-\mu_X)$,
and we have $X\leq_{0\textup{-tvar-rl}} Y$.

A particular case of the generalized Pareto distributions is the original Pareto distibution. Given a random variable $X\sim P(a,k)$, we have that
\[
\overline F(x) = \left\{\begin{array}{ll}
1,& \text{ for all }x\leq k,\\
\left(\dfrac{k}{x} \right)^{a},& \text{ for all }x>k.
\end{array}\right.
\]
Given that, in this case, $u_X=+\infty$, we need to compute $\overline F_t$, and $F^{-1}_t$, for all $t\geq 0$, and the following cases are distinguished:
\begin{itemize}
\item[(i)] If $t<k$, then
\[
\begin{array}{lll}
\overline F_t(x) & = & \left\{\begin{array}{ll}
1,& \text{ for all } x\leq k-t,\\
\left(\dfrac{k}{t+x} \right)^{a},& \text{ for all }x>k-t.
\end{array}\right.
\end{array}
\]
\[
F^{-1}_t(p) = \dfrac{k}{(1-p)^{1/a}} -t, \text{ for all }p\in(0,1).
\]

\item[(ii)] If $t\geq k$, then
\[
\begin{array}{lll}
\overline F_t(x) & = & \left\{\begin{array}{ll}
1,& \text{ for all }x\leq 0,\\
\left(\dfrac{t}{t+x} \right)^{a},& \text{ for all }x>0.
\end{array}\right.
\end{array}
\]
\[
F^{-1}_t(p)  = \dfrac{t}{(1-p)^{1/a}} -t, \text{ for all }p\in(0,1).
\]
\end{itemize}

Let $X\sim P(a_X,k_X)$ and $Y\sim P(a_Y,k_Y)$ with survival functions $\overline F$ and $\overline G$, respectively, such that $a_X,a_Y>1$ (in order to have finite means). Let us suppose that $\frac{a_X k_X}{a_X-1}>\frac{a_Y k_Y}{a_Y-1}$ (in order to $E[X]>E[Y]$).

If $t<k_X,k_Y$,
\begin{eqnarray*}
H_t(p)&=&\dfrac{k_Y}{(1-p)^{1/a_Y}}-\dfrac{k_X}{(1-p)^{1/a_X}},\text{ for all }p\in[0,1),\\
H_t(0)&=& k_Y-k_X.
\end{eqnarray*}

If we suppose that $k_X>k_Y$, then $H_t(0)<0$. Moreover, if $a_X>a_Y>1$, then
\[
\lim_{p\rightarrow 1^-}H_t(p)=+\infty.
\]
Then, $S^-\left(H_t\right)\geq 1$ where the last sign change occurs from - to + with crossing point $p_t$. Under the notation in Theorem \ref{teocondsuf}, $t\in T_2$. 

If $k_Y<t<k_X$,
\begin{eqnarray*}
H_t(p)&=&\dfrac{t}{(1-p)^{1/a_Y}}-\dfrac{k_X}{(1-p)^{1/a_X}},\text{ for all }p\in[0,1),
\end{eqnarray*}
the situation is similar to the previous one, and $t\in T_2$.

If $k_Y<k_X<t$, and we suppose $a_X>a_Y>1$, we have
\begin{eqnarray*}
H_t(p)&=&\dfrac{t}{(1-p)^{1/a_Y}}-\dfrac{t}{(1-p)^{1/a_X}}>0,
\end{eqnarray*}
for all $p\in[0,1)$. Therefore, $t\in T_1$.

To sum up, if
\begin{itemize}
\item[(i)] $\frac{a_X k_X}{a_X-1}>\frac{a_Y k_Y}{a_Y-1}$,
\item[(ii)] $a_X>a_Y>1$, and
\item[(iii)] $k_X>k_Y$,
\end{itemize}
then $T=T_1\  \cup \ T_2$ and, from Theorem \ref{teocondsuf}, $X\leq_{p_0\textup{-tvar-rl}}Y$, for some $p_0\in(0,1)$. Obviously, for large values of $a_X$ and $a_Y$ condition (i) reduces to condition (iii).

It is important to note that, in this case, it is possible to obtain explicit expressions for the exact value of $p_0$. In particular, $X\leq_{p_0\textup{-tvar-rl}}Y$ for
\begin{equation}\label{p0pareto}
p_0=1- \left[\frac{a_X(a_Y-1)k_X}{a_Y(a_X-1)k_Y}\right]^{a_X a_Y/(a_Y-a_X)}.
\end{equation}
\end{example}

\begin{example}[Loglogistic distributions]\label{logistic1}
The loglogistic distribution is widely used to fit finance and risk data as the duration of claim for income protection insurance, natural catastrophe claims and, of course, insurance claims. This distribution is closely related to the logistic one which is widely used in statistical modeling, for example when we consider extreme data (see Balakrishnan, 1992, for further details). A random variable follows a logistic distribution, denoted by $X\sim \textup{Logistic}(\mu,\sigma)$, with parameters $\mu\in \mathbb R, \sigma>0$, if its survival function is given by 
\begin{eqnarray*}
\overline F(x) &=&\dfrac{1}{1+\exp\left(\frac{x-\mu}{\sigma}\right)},\text{ for all }x\in\mathbb R.
\end{eqnarray*}

If $X\sim \textup{Logistic}(\mu,\sigma)$, it holds that
\begin{eqnarray*}
\overline F_t(x) &=& \frac{1+\exp\left(\frac{t-\mu}{\sigma}\right)}{1+\exp\left(\frac{x+t-\mu}{\sigma}\right)},\text{ for all }t\in\mathbb R,x\in\mathbb R,\\
F_t^{-1}(p) &=& \sigma\log\left[\frac{p+\exp\left(\frac{t-\mu}{\sigma}\right)}{1-p}\right]+\mu-t,\text{ for all }t\in\mathbb R,p\in(0,1).
\end{eqnarray*}

Given $X\sim \textup{Logistic}(\mu_X,\sigma_X)$ and $Y\sim \textup{Logistic}(\mu_Y,\sigma_Y)$, it follows that
\begin{eqnarray}\label{logisticht}
H_t(p)&=&G_t^{-1}(p)-F_t^{-1}(p)\nonumber\\
&=&\sigma_Y\log\left[\frac{p+\exp\left(\frac{t-\mu_Y}{\sigma_Y}\right)}{1-p}\right]
-\sigma_X\log\left[\frac{p+\exp\left(\frac{t-\mu_X}{\sigma_X}\right)}{1-p}\right]+(\mu_Y-\mu_X),
\end{eqnarray} 
for all $p\in(0,1)$ and $t\in\mathbb R$. 

Next, we consider two situations where $X$ and $Y$ share one of the parameters.
\begin{itemize}
\item[(i)] If $\sigma_X=\sigma_Y:=\sigma>0$, then
\begin{eqnarray*}
H_t(0) &=& 0,\\
H'_t(p) &=& \frac{\sigma}{p+\exp\left(\frac{t-\mu_Y}{\sigma}\right)}-\frac{\sigma}{p+\exp\left(\frac{t-\mu_X}{\sigma}\right)}.
\end{eqnarray*}
If we assume that $\mu_X<\mu_Y$, then $H_t(p)\geq 0$, for all $p\in(0,1)$ and $t\in\mathbb R$. From Proposition \ref{hr}, $F_t^{-1}(p)\leq G_t^{-1}(p)$, for all $p\in(0,1)$.

\item[(ii)] If $\mu_X=\mu_Y:=\mu\in\mathbb R$, \eqref{logisticht} can be written as
\begin{eqnarray*}
H_t(p)&=&\sigma_Y\log\left[\frac{p+\exp\left(\frac{t-\mu}{\sigma_Y}\right)}{1-p}\right]
-\sigma_X\log\left[\frac{p+\exp\left(\frac{t-\mu}{\sigma_X}\right)}{1-p}\right],
\end{eqnarray*} 
for all $p\in(0,1)$ and $t\in\mathbb R$.

Let us consider
\begin{equation}\label{logistic}
h(\sigma)=\sigma\log\left[\frac{p+\exp\left(\frac{t-\mu}{\sigma}\right)}{1-p}\right],
\end{equation}
for all $p\in(0,1)$ and $t\in\mathbb R$. Given that its derivative is given by
\[
h'(\sigma)=\log\left[\frac{p+\exp\left(\frac{t-\mu}{\sigma}\right)}{1-p}\right]-\left(\frac{t-\mu}{\sigma}\right)\frac{\exp\left(\frac{t-\mu}{\sigma}\right)}{p+\exp\left(\frac{t-\mu}{\sigma}\right)},
\]
it is easy to see that \eqref{logistic} is increasing if, and only if,
\[
\log\left[p+\exp\left(\frac{t-\mu}{\sigma}\right)\right]\geq \left(\frac{t-\mu}{\sigma}\right)\frac{\exp\left(\frac{t-\mu}{\sigma}\right)}{p+\exp\left(\frac{t-\mu}{\sigma}\right)}+\log(1-p),
\]
or, equivalently,
\begin{multline*}
p\log\left[p+\exp\left(\frac{t-\mu}{\sigma}\right)\right]+\exp\left(\frac{t-\mu}{\sigma}\right)\log\left[p+\exp\left(\frac{t-\mu}{\sigma}\right)\right]\\ \geq \left(\frac{t-\mu}{\sigma}\right)\exp\left(\frac{t-\mu}{\sigma}\right)+p\log(1-p)+\exp\left(\frac{t-\mu}{\sigma}\right)\log(1-p).
\end{multline*}

Let us fix $t\in\mathbb R$ and $p>1-\exp\left(\frac{t-\mu}{\sigma}\right)$ and let
\begin{eqnarray*}
A &:=& p\log\left[p+\exp\left(\frac{t-\mu}{\sigma}\right)\right]\\
B &:=& \exp\left(\frac{t-\mu}{\sigma}\right)\log\left[p+\exp\left(\frac{t-\mu}{\sigma}\right)\right]\\
C &:=& \left(\frac{t-\mu}{\sigma}\right)\exp\left(\frac{t-\mu}{\sigma}\right)\\
D &:=& p\log(1-p)+\exp\left(\frac{t-\mu}{\sigma}\right)\log(1-p)
\end{eqnarray*}

In this case, we have that
\begin{eqnarray*}
A&\geq& D \\
B&\geq& C 
\end{eqnarray*}
and, therefore, $h'(\sigma)\geq 0$. In particular, note that $h$ is increasing for all $p\in(0,1)$ if $t\geq \mu$ and, consequently, $H_t(p)\geq 0$, for all $p\in(0,1)$. For the rest of values of $t$ it is necessary to study the corresponding crossing points. 

Let us consider $X\sim \textup{Logistic}(0,2)$ and $Y\sim \textup{Logistic}(0,5)$ as a particular example. In this case, we use the graphical tool described in \eqref{bivariateplot} which simplifies the analysis. In Figure \ref{logisticfigure}, we can see the corresponding function $H_t(p)$ over $p\in(0,1)$ and $t\in(-20,20)$. It is easy to see that for $t>0$, $H_t(p)\geq 0$, for all $p\in (0,1)$, while for $t<0$ the function suffers a change of sign from - to + in $p_t$. 
\begin{figure}
\centering
\includegraphics[width=0.5\textwidth]{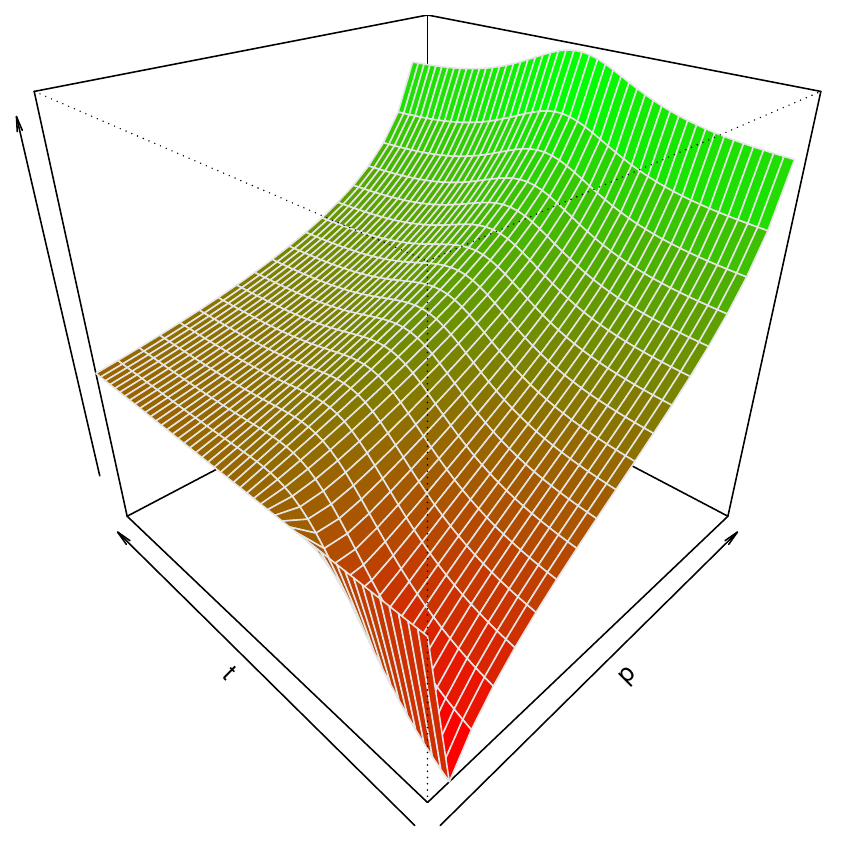}
\caption{$H_t(p)$ for $X\sim \textup{Logistic}(0,2)$ and $Y\sim \textup{Logistic}(0,5)$ with $t\in(-20,20)$ and $p\in(0,1)$.}
\label{logisticfigure}
\end{figure}

In order to analyze the $p_t$ values we can compute them numerically for different values of $t\in\mathbb R$. In Figure \ref{plogisticfigure} we can see the $p_t$ values for $t\in(-100,100)$ and we realize that this crossing point is always smaller than 1/2. In fact, $\lim_{t\rightarrow -\infty}p_t=1/2$. 
\begin{figure}
\centering
\includegraphics[width=0.5\textwidth]{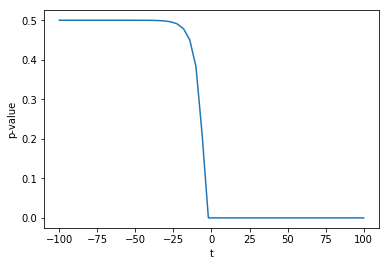}
\caption{$H_t(p)$ for $X\sim \textup{Logistic}(0,2)$ and $Y\sim \textup{Logistic}(0,5)$ with $t\in(-20,20)$ and $p\in(0,1)$.}
\label{plogisticfigure}
\end{figure}
Given that $T=T_1\ \cup\ T_2$, we can say that in this case $X\leq_{0.5\textup{-tvar-rl}}Y$.

\end{itemize}

The connection between the logistic and the loglogistic distributions is the following. If $X\sim \textup{Logistic}(\mu,\sigma)$, then 
\begin{equation*}\label{loglogistic}
X^*=\exp(X) \sim LogLogistic(e^\mu,1/s).
\end{equation*}
Therefore, if $X$ and $Y$ are two logistic distributions such that
\[
X\leq_{p_0\textup{-tvar-rl}} Y,
\]
 taking into account that the exponential function is increasing and convex and using this fact in Theorem \ref{closuretrans}, it holds that 
 \[
 X^*=\exp(X)\leq_{p_0\textup{-tvar-rl}} \exp(Y)=Y^*,
 \]
  where $X^*$ and $Y^*$ follow two loglogistic distributions.
\end{example}

\begin{example}[Distorted distributions]\label{distorted1}
Distorted distributions were introduced by Denneberg (1990) and Wang (1995, 1996) in the context of actuarial science for several variety of insurance problems. A distortion function is a  continuous, nondecreasing and piecewise
differentiable function $h : [0, 1] \rightarrow [0, 1]$ such that
$h(0) = 0$ and $h(1) = 1$. Given a random variable $X$ with finite mean and right endpoint of its support $u_X$ and a distortion function $h$, the distorted random variable $X_h$ induced by $h$ has survival function given by
\begin{equation*}
\ov F_{h}(x)=h\left(\ov F(x)\right),
\end{equation*}
for all $x$ in the support of $X$, denoted by $\textup{Supp(X)}$. 

If $X_h$ is the distorted random variable induced by $h$, it holds that 
\begin{eqnarray*}
\ov F_{h,t}(x)&=&\frac{h(\ov F(x+t))}{h(\ov F(t))},\text{ for all }t<u_X,x\in\textup{Supp(X)},\\
F^{-1}_{h,t}(p)&=&F^{-1}\left(1-h^{-1}\left[(1-p)h\left(\ov F(t)\right)\right]\right),\text{ for all }t<u_X,p\in(0,1).
\end{eqnarray*}
 
Given two distorted random variables $X_h$ and $X_l$ induced 
by $h$ and $l$, respectively, it follows that
\begin{eqnarray*}
H_t(p)&=&F^{-1}_{l,t}(p)-F^{-1}_{h,t}(p)\\
&=& F^{-1}\left(1-l^{-1}\left[(1-p)l\left(\ov F(t)\right)\right]\right)-F^{-1}\left(1-h^{-1}\left[(1-p)h\left(\ov F(t)\right)\right]\right),
\end{eqnarray*}
for all $p\in(0,1)$ and $t<u_X$.

On one hand, if $h(t)\geq l(t)$, for all $t\in(0,1)$, then $F^{-1}_{h,t}(p)\leq F^{-1}_{l,t}(p)$, for all $t<u_X$ and all $p\in(0,1)$, and $X_{h,t}\leq_{0\textup{-tvar-rl}} X_{l,t}$. For example, let $h(t)=t^\alpha$ and $l(t)=t^\beta$, for $0<\alpha<\beta$, respectively. In this case, it holds that $h(t)\geq l(t)$, for all $t\in(0,1)$ and, consequently, $X_{h}\leq_{0\textup{-tvar-rl}} X_{l}$.

A more interesting case is when $h(t)\not\geq l(t)$, for all $t\in(0,1)$. From Theorem \ref{teocondsuf}, $T=T_1\ \cup\ T_2$ if, and only if, for each $t<u_X$, one of the following conditions holds:
\begin{itemize}
\item[(i)] $F^{-1}_{h,t}(p)\leq F^{-1}_{l,t}(p)$, for all $p\in(0,1)$, or
\item[(ii)] $S^-\left(H_{t}\right)\geq 1$ where the last sign change occurs from - to + with crossing point $p_t \in (0, 1)$.
\end{itemize} 
Let $f_{h,l}^q(p)=h^{-1}((1-p)h(q))-l^{-1}((1-p)l(q))$, for $q\in(0,1)$, then (ii) is equivalent to 
\[
S^-\left(f_{h,l}^q\right)\geq 1,
\]
 where the last sign change occurs from - to + with crossing point $p_t \in (0, 1)$. Taking into account Remark \ref{remark3}, we can plot $f_{h,l}^q$ over $p\in(0,1)$ and $q\in(0,1)$ and observe the sign changes over $q\in (0,1)$. For example, in Figure \ref{distortedfigure}, we show the plot of $f_{h,l}^q(p)$ for $h(t)=t^{0.45}$ and $l(t)=1-(1-t)^{2.75}$ and we can see that there are some $q$'s for which this function is always positive (this case corresponds to $t\in T_1$) and some other $q$'s for which there is a sign change. Therefore, $T=T_1\ \cup\ T_2$.
\begin{figure}
\centering
\includegraphics[width=0.5\textwidth]{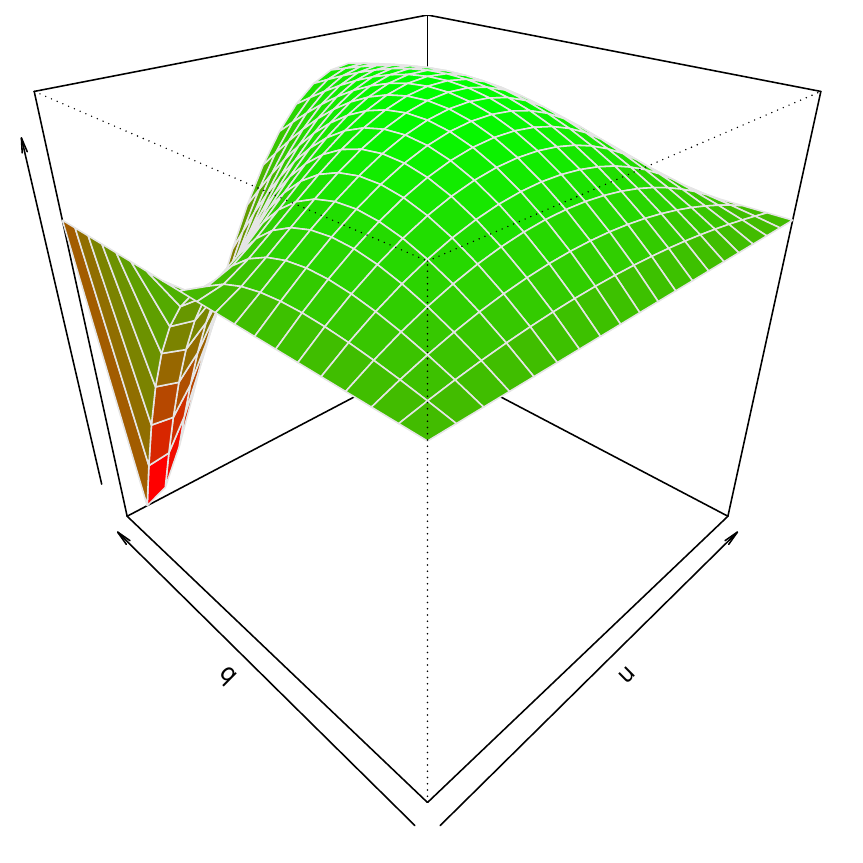}
\caption{$f_{h,l}^q(p)$ for $h(t)=t^{0.45}$ and $l(t)=1-(1-t)^{2.75}$, for all $p\in(0,1)$ and $q\in(0,1)$.}
\label{distortedfigure}
\end{figure}
\end{example}

\section{Real data example}\label{sec:5}

In this example of methodological nature we illustrate the previous ideas and results through a real insurance dataset. \textsf{AutoClaims} is a dataset included in the package \textsf{insuranceData} (\url{www.r-project.org}, R Core Team, Vienna, Austria) and it contains the claims experience from a large midwestern (US) property and casualty insurer for private passenger
automobile insurance. The dependent variable is the amount paid on a closed claim, in (US) dollars
(claims that were not closed by year end are handled separately) and  policyholders are categorized
according to a risk classification system based on a rating class of operator depending on age, gender, marital status, use of vehicle.  Next, we focus on the classes \textsf{C1A}, \textsf{C1B}, \textsf{C71}, \textsf{C72}, \textsf{C7A}, \textsf{C7C}, \textsf{F11} and \textsf{F71}. From now on, let $Y$ be the corresponding claims for the \textsf{F11} class. Our objective is to analyze if $X\leq_{p_0\textup{-tvar-rl}}Y$ where $X$ represents the corresponding claims for any of rest of classes.

First, let us see that $X\nleq_{hr} Y$. It is known that the definition of the hazard rate order can be rewritten in terms of a property of the plot of the two survival functions. A subset $A$ of the Euclidean space is called star-shaped with respect to a point $s$ if for every $x \in A$ we have that $A$ contains the whole line segment between $x$ and $s$. A real function $f$ is called star-shaped with respect to a point $(a, b)$, if its epigraph is star-shaped with respect to $(a, b)$. Now, if we consider the plot of the points $(\overline F (t), \overline G(t))$ (in short $\overline P - \overline P$ plot), we have that $X \leq_{hr} Y$, if and only if, the $\bar P - \bar P$ plot is star-shaped with respect to $(0, 0)$ (see M\"uller and Stoyan, 2002, p. 9). In Figure \ref{examplefigure} we give the empirical $\overline P - \overline P$ plots of the samples, that is, the plot of the points $(\overline F_n(t), \overline G_n(t))$, where $\overline F_n$ and $\overline G_n$ are the empirical
survival functions of $X$ and $Y$, respectively. From the $\overline P - \overline P$ plot, it is clear that the hazard rate order does not hold. Now, it is natural to wonder if $X\leq_{p_0\textup{-tvar-rl}}Y$ for any $p_0\in (0,1)$.
\begin{figure}
\centering
\includegraphics[width=0.75\textwidth, trim=0cm 0.6cm 0cm 0.6cm]{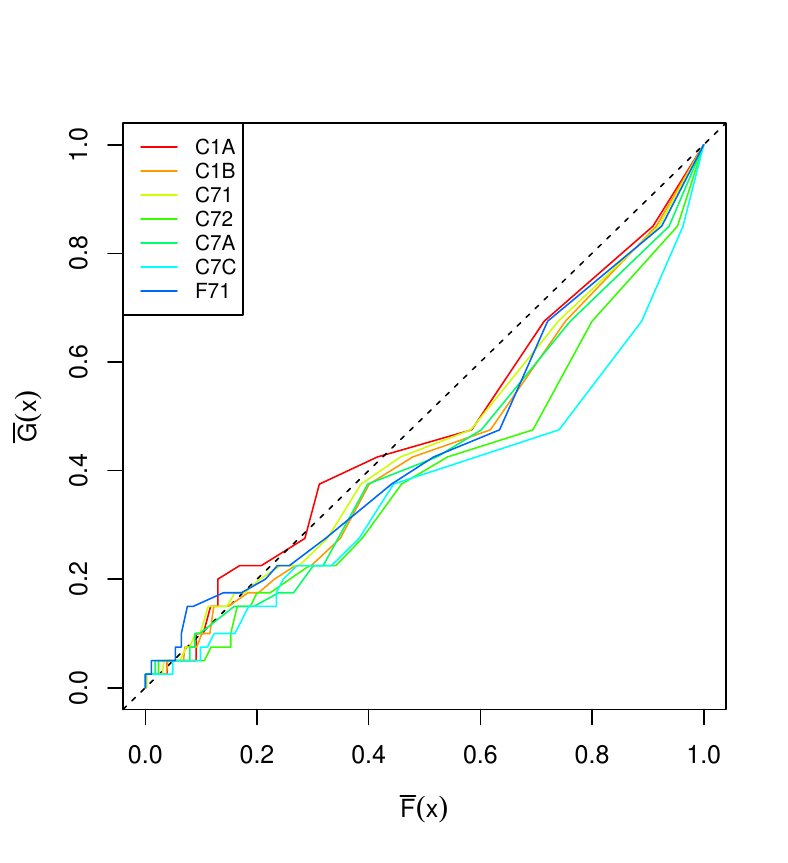}
\caption{Empirical $\overline P - \overline P$ plot.}
\label{examplefigure}
\end{figure}

As mentioned by Arnold (2005, p. 17) and references therein, the potential fields of application for Pareto and/or Pareto related models includes modeling claim premiums. One of the most used Pareto distributions is the classical one, mainly due to its small number of parameters.  Recall that in Example \ref{pareto2} we provide conditions for the comparison between two classical Pareto distributions. In particular, if $X\sim P(a_X,k_X)$ and $Y\sim P(a_Y,k_Y)$, $X\leq_{p_0\textup{-tvar-rl}}Y$ whenever $a_X>a_Y>1$, $k_X>k_Y$ and $a_X k_X/(a_X-1)>a_Y k_Y/(a_Y-1)$.

Taking into account these ideas, first we fit a classical Pareto distribution for each of the classes by using the \textsf{R}-package \textsf{fitdistrplus}. The computed parameters as well as the corresponding values of $ak/(a-1)$ are shown in Table \ref{table1}. From the value of the corresponding Kolmogorov-Smirnov statistic (see last column in Table \ref{table1}), we can assume that our datasets follow a Pareto distribution. It is worth to mention that other distributions could offer a greater goodness of fit, but the aim of this example is to use our previous results. 
\begin{table}
\begin{center}
\begin{tabular}{ccccc}
& \multicolumn{2}{c}{Parameters} & & Test statistic\\ \cline{2-3} \cline{5-5}
 & $k$  & $a$ &  $ak/(a-1)$ & K-S \\ \hline
\textsf{C1A} & $4413.1532$ &  $3.5435$ & 6148.175 & 0.1239\\
\textsf{C1B} & $7360.4283$ &  $4.8540$ & 9270.204 & 0.0893\\
\textsf{C71} & $7204.7579$ &  $5.0193$ &  8997.277 & 0.0974\\
\textsf{C72} & $10548.676$ &  $5.8600$ & 12719.15 & 0.1260\\ 
\textsf{C7A} & $20830.5147$ &  $11.9036$ & 22740.94 & 0.1068\\
\textsf{C7C} & $17011.9103$ &  $8.7029$ & 19220.39 & 0.1339\\ 
\textsf{F11} & $2655.6875$ &  $2.3717$ & 4591.691 & 0.1465\\
\textsf{F71} & $68547.290$ &  $43.8327$ & 70147.64 & 0.0875\\ \hline
\end{tabular}
\end{center}
\caption{Fitted classical Pareto models.}\label{table1}
\end{table}

Finally, it is easy to see that these values satisfy the conditions for $X\leq_{p_0\textup{-tvar-rl}}Y$ where $X$ represents the claims of \textsf{C1A}, \textsf{C1B}, \textsf{C71},  \textsf{C72}, \textsf{C7A}, \textsf{C7C} or \textsf{F71} and $Y$ represents the corresponding claims of \textsf{F11}. 
Besides, Table \ref{table2} shows the $p_0$ values for which it holds $X\leq_{p_0\textup{-tvar-rl}}Y$ for each pair of claims, according to \eqref{p0pareto}.


\begin{table}
\begin{center}
\begin{tabular}{cccccccc}
& \textsf{C1A} & \textsf{C1B} & \textsf{C71} & \textsf{C72} & \textsf{C7A} & \textsf{C7C} & \textsf{F71}\\ \hline
$p_0$ &  0.8767 & 0.9615  & 0.9513 & 0.9827 & 0.9912 & 0.9906 & 0.9999\\ \hline
\end{tabular}
\end{center}
\caption{$p_0$ values.}\label{table2}
\end{table}

\section{Conclusions and some comments}\label{sec:6}

Comparing assets in terms of their Value at Risk (VaR) or Tail Value at Risk (TVaR) is one of the main tools in insurance in order to analyze the perform of risk assesment. Bank risk managers follow the Basel Committee on Banking Supervision  recommendations that recently promoted shifting the quantitative risk metrics system from
VaR to TVaR. In particular, the use of the Tail Value at Risk (TVaR) offers more representative information on risks. However,  these comparisons do not consider the whole risk and, in some situations, they are not necessary for all $p$ (the focus is usually on large values of $p$). 

In this paper, we have proposed a new family of stochastic orders indexed by $p_0\in(0,1)$, that compares the tail values at risk of two conditional random variables from a certain probability $p_0$ and beyond. More precisely, given two random variables $X$ and $Y$, it is said that $X$ is smaller than $Y$ in the $p_0\textup{-tvar-rl}$ order, denoted as $X\leq_{p_0\textup{-tvar-rl}}Y$, if
 \[
\mathrm{TVaR}[X_t;p]\leq \mathrm{TVaR}[Y_t;p], \text{ for all } t<u_X,u_Y, \text{ and }p\in[p_0,1),
\]
where $X_t$ and $Y_t$ are the residual lives of $X$ and $Y$, respectively. 

For two random risks $X$ and $Y$, previous comparisons are of interest, for instance, when insurance companies go to reinsurance companies to protect their capital against possible large claims via a stop loss contract or when an insurance company offers a contract with some franchise deductible. Besides, it is useful in medical research where the remaining lifetime of the patients is considered. 

According to Denuit et al. (2005, p. 89), the Wang risk measure, $C_g $, of a random variable $ X $ is defined as
\[
C_g(X)=\int_0^{+\infty}g(\ov {F}(x))dx,
\]
where $g$ is a nondecreasing distortion such that $g(0)=0$ and $g(1)=1$. The  increasing convex order (stop-loss), or equivalently, the comparison of the tails of the Value at Risk, can be characterized by comparing the Wang risk measures. More specifically, we have that
\[
TVaR[X,p] \leq TVaR[Y,p], \textup{ for all } p \in (0,1) \textup{ if, and only if, } C_g(X) \leq C_g(Y),
\]
for any concave distortion $g$ (see, for example, Denuit et al., 2005, pp. 152-154, Hong, Karni and Safra, 1987, and Wang and Young, 1998). Unfortunately, we cannot extend this characterization when comparing the $TVaR$'s from $p_0$ on.

Along this work, we have shown relationships between the new order and some other stochastic comparisons of interest in risk theory. In particular, it is worth to mention that it is closely related to the $p\textup{-rl}$ and the mrl orders. In addition, we have provided some preservation properties such as the closure under convergence and increasing convex transformations, and  sufficient conditions for this order to hold. Finally, some theoretical examples dealing with well-known risk models like the Pareto, the loglogistic and the distorted distributions, and a real-data example were described from a methodological perspective.

\end{document}